\documentclass[]{interact}
\usepackage{mathrsfs}
\usepackage{graphics}
\usepackage{rotating}
\usepackage{multirow}
\usepackage{xcolor}
\usepackage{lscape}
\usepackage{epstopdf}
\usepackage[caption=false]{subfig}
\theoremstyle{plain}
\newtheorem{theorem}{Theorem}[section]
\newtheorem{lemma}[theorem]{Lemma}
\newtheorem{corollary}[theorem]{Corollary}
\newtheorem{proposition}[theorem]{Proposition}
\theoremstyle{definition}
\newtheorem{definition}[theorem]{Definition}
\newtheorem{example}[theorem]{Example}
\theoremstyle{remark}

\newcommand{\rg}{{\mathscr{R}}}
\newcommand{\nl}{{\mathscr{N}}}

\newcommand{\mc}[1]{\mathcal {#1}}
\newcommand{\dg}{{\dagger}}
\newcommand{\bcr}{{\texttt{bcirc}}}
\newcommand{\uf}{{\texttt{unfold}}}
\newcommand{\fo}{{\texttt{fold}}}
\newcommand{\rank}{{\text{rank}}_{t}}
\newcommand{\n}{{*_N}}

\newcommand{\cpqn}{{\mathbb{C}}^{p\times q\times n}}
\newcommand{\cpkn}{{\mathbb{C}}^{p\times k\times n}}
\newcommand{\ckqn}{{\mathbb{C}}^{k\times q\times n}}
\newcommand{\ckpn}{{\mathbb{C}}^{k\times p\times n}}
\newcommand{\cqkn}{{\mathbb{C}}^{q\times k\times n}}
\usepackage{algpseudocode}
\usepackage{algorithm}

\begin{document}


\title{Computation of outer inverse of tensors based on $t$-product}

\author{
\name{Ratikanta Behera,\textsuperscript{a} Jajati Keshari Sahoo\textsuperscript{b} and Yimin Wei\textsuperscript{c}\thanks{\textsuperscript{a}Email: ratikanta@iisc.ac.in,~\textsuperscript{b}Email: jksahoo@goa.bits-pilani.ac.in\\~\textsuperscript{c}Corresponding Author (Yimin Wei). Email: ymwei@fudan.edu.cn}}
\affil{\textsuperscript{a} Department of Computational and Data Sciences, Indian Institute of Science, Bangalore, Karnataka, India \\\textsuperscript{b} Department of Mathematics, BITS Pilani, K.K. Birla Goa Campus, Zuarinagar, Goa, India\\\textsuperscript{c}  School of Mathematics Sciences and Shanghai Key Laboratory of Contemporary Applied Mathematics, Fudan University, Shanghai 200433, P.R. China}
}

\maketitle

\begin{abstract}
Tensor operations play an essential role in various fields of science and engineering, including multiway data analysis. In this study, we establish a few basic properties of the range and null space of a tensor using block circulant matrices and the discrete Fourier matrix. We then discuss the outer inverse of tensors based on $t$-product with a prescribed range and kernel of third-order tensors. We address the relation of this outer inverse with other generalized inverses, such as the Moore-Penrose inverse, group inverse, and Drazin inverse. In addition, we present a few algorithms for computing the outer inverses of the tensors. In particular, a $t$-QR decomposition based algorithm  is developed for computing the outer inverses. 
\end{abstract}


\section{Introduction and motivations}
A tensor is a generalization of a matrix and vectors, and has become an active area of research in various fields, including signal processing \cite{huang2019tensor}, machine learning \cite{Han2023T,sidiropoulos2017tensor}, and analytical chemistry \cite{kroonenberg2008applied}. Furthermore, tensor theories are encountered in a number of fields of practical interest, such as the spectral theory of hypergraphs \cite{ChenQiZhang17, CooperDu12}, eigenvalues of tensors \cite{BuWeiLi15,Che2017pseudo,Ding2015generalized, Ma2023noda,QiLuo17}, eigenvalues of nonnegative tensors \cite{ChangKelly11, CheBuQiWei}, tensor inversion, solutions of multilinear systems \cite{BraliNT13,Ding2016solving,sun2018,Wang2022predefined} and tensor complementarity problems
\cite{Che2019stochastic,Wang2022randomized,Wang2023fixed,Wei2023neural}. However, tensors based on different tensor products have been frequently investigated in recent years, including Einstein products \cite{ein, BraliNT13,Ji2018Drazin,Ji2020outer}, {\it t}-products \cite{Braman10, Kilmer13, Martin13}, Shao products \cite{Shao13}, and $n$-mode products \cite{BaderKolda07}. This study focuses on one such class, tensors in the framework of {\it t}-products, introduced by Kilmer and Martin \cite{Braman10}, which has significantly affected many areas of science and engineering (e.g., image and signal processing \cite{KimNing13}, robust tensor PCA \cite{liu2018improved}, computer vision \cite{HaoBraman13}, tensor computations \cite{Che2022an,Chen2023perturbation,Chen2022tensor,Mo2020time,Wang2023solving,Wang2023fixed,Wei2023neural}, date completion and denoising \cite{ZhengYong21, HuYuan17}). The {\it t}-product takes advantage of the fact that  high-order tensors transform, with the help of the discrete Fourier matrix,  into two-dimensional arrays (matrix) via a recursive unfolding process, and tensors enjoy a linear algebraic-like structure.

The purpose of this study is to use a linear algebraic-like structure concept to compute the outer inverse of tensors with a prescribed range and/or kernel \cite{Stanimirovi2022new,Wei1998characterization}. It is worth recalling the outer inverse using matrix. According to standard notation, $\mathbb{K}^{m\times n}$ represents the set of $m \times n$ real ($\mathbb{R}^{m\times n}$) or complex ($\mathbb{C}^{m\times n}$) matrices. A matrix $ X \in \mathbb{C}^{m\times n}$ is an outer  (respectively inner) inverse of  $ A \in \mathbb{C}^{n\times m}$, if the equality $XAX = X$ (respectively $AXA = A$) is satisfied. Let ${A}\{\lambda\}$ be the set of all the outer inverses of ${A}$, that is, ${A}\{\lambda\} = \{ X | XAX= X\}$. The main application of the outer inverses is to solve multilinear systems. Such an application of outer inverses generalizes the application of ordinary inverses. One can find Several representations of outer inverses in \cite{Ben-Israel} for efficient computations and properties.  Furthermore, if a matrix ${X}$ from the set ${A}\{\lambda\}$  whose range and kernel satisfy $R({X})=R({B})$, respectively, $N({X}) = N({C})$ is denoted by ${A}^{(\lambda)}_{R({B}),*}$ (resp. ${A}^{(\lambda)}_{*,N({C})}$). For detailed expositions of  the generalized inverses of matrices, including the computational aspects and applications, we refer to the following  books: Ben-Israel and Greville \cite{Ben-Israel}, Rao and Mitra \cite{raobook},  and Campbell and Meyer \cite{Campbell91}.  Wei et al. \cite{Weibook}, and Nashed \cite{nashed2014generalized, Nashed87}.

In this paper, we focus on the outer inverses of third-order tensors with prescribed ranges and/or kernels and present effective algorithms using $t$-QR decomposition for computations. Outer inverses involve many well-known generalized inverses such as the Moore-Penrose inverse \cite{Cong2022acute,Miao2020generalized}, the Drazin \cite{miao2020t}, Core-EP inverse \cite{Cong2022characterization,Liu2021perturbation}, and the group inverse. An additional motivation for investigating outer inverses is to compute efficiently.

The global organization of this study is briefly described as follows.  Section 2 is devoted to tensor notations and definitions along with a review of existing results.  Section 3 discusses the computation of the outer inverses of tensors with a prescribed range and/or kernel of third-order tensors along with effective algorithms for computing the outer inverses of tensors using $t$-QR decomposition. Finally, concluding observations and a preview of the possibilities for further work are presented in Section 4.

\section{Preliminaries}\label{SecPreliminaries}

Let $\mc{S} \in \mathbb{C}^{p\times q\times n}$  and $\mc{T} \in \mathbb{C}^{q\times l\times n}$ be two tensors. For $k=1,2,\cdots n$ the frontal faces of the above tensors is defined as $S_{k} = \mc{S}(:, :, k) \in  \mathbb{C}^{p\times q}$ and $T_{k} = \mc{T}(:, :, k) \in  \mathbb{C}^{q\times l}$. The authors of \cite{kilmer11} are defined the operations of $\bcr: \mathbb{C}^{p \times q\times n} \mapsto  \mathbb{C}^{pn\times qm}$ , 
$\uf: \mathbb{C}^{p \times q\times n} \mapsto  \mathbb{C}^{pq\times n}$ 
and $\fo$ operations as follows:

\begin{equation*}
    {\bcr}\left(
     \mc{T}
\right) := \begin{pmatrix}
{T}_{1} & {T}_{n} &\cdots & {T}_{2}\\
{T}_{2} & {T}_{1}    & \cdots & {T}_{3}\\
\vdots  &   \vdots  &    \ddots & \vdots\\
{T}_{n}& {T}_{n-1}&  \cdots & {T}_{1}
 \end{pmatrix}, ~~
{\uf}\left(
     \mc{T}
\right) := \begin{pmatrix}
{T}_{1} \\
{T}_{2} \\
\vdots  \\
{T}_{n}
 \end{pmatrix}, 
\end{equation*}
following the above definition, it is easy to verify 
\begin{equation*}
\fo(\uf(\mc{T}))=\mc{T} =\bcr^{-1}(\bcr(\mc{T})),
\end{equation*}
where $\fo$ is the inverse operation of $\uf$ and $\bcr^{-1}$ is the inverse operation of $\bcr$, i.e., $\bcr^{-1}: \mathbb{C}^{pn\times qm} \mapsto  \mathbb{C}^{p \times q\times n}$.  Now product of two tensors is defined \cite{Martin13} with the help of `$\bcr$' and `$\fo$'.
\begin{definition}
    Let  $\mc{S} \in \mathbb{C}^{p\times q\times n}$ and $\mc{T} \in \mathbb{C}^{q\times l\times n}$ be two tensors. The product of two tensors defined by $\mc{C} = \mc{S}*\mc{T} \in \mathbb{C}^{p\times l\times n}$ such that
  \begin{equation*}
      \fo(\bcr(\mc{S})\cdot \uf(\mc{T})) = \mc{C} =
      \bcr^{-1}(\bcr(\mc{S})\cdot \bcr(\mc{T})), 
      \end{equation*}
      where `$\cdot$' is the simple matrix multiplication.
\end{definition}

It is worth mentioning that the block circulant matrices can be diagonalized with the help of normalized discrete Fourier transform (DFT) matrix. Specifically, for 
$\mc{T} \in \mathbb{C}^{p\times q \times n}$, then with the help of $\bcr(\mc{T}) \in \mathbb{R}^{np\times nq \times n}$  and DFT once can write
\begin{equation}\label{FFTEq}
    (F_{n} \otimes I_{n})\cdot 
    {\bcr}
     (\mc{T}) \cdot (F_{n}^* \otimes I_{n})=
    \begin{pmatrix}
     {D}_1 &  &  & \\
     & {D}_2 &  & \\
     & &  \ddots &\\ 
    & &  & {D}_{n}
    \end{pmatrix}=\textbf{fft}(\mc{T}, [~], 3)
\end{equation}
where $\textbf{fft}(\mc{T}, [~], 3)$ is fast Fourier transform in third direction and $F_{n}$ is the $n\times n$ DFT matrix defined as 
\begin{equation*}
   F_{n}=\frac{1}{\sqrt{{n}}} \begin{bmatrix}
    1 & 1 & 1 & \cdots & 1\\
    1 & w & w^2 &\cdots& w^{{n}-1} \\
    1 & w^2 & w^4 &\cdots & w^{2({n}-1)}\\
    \vdots & \vdots & \vdots& \ddots & \vdots \\
    1 & w^{({n}-1)} & w^{2({n}-1)} &\cdots& w^{({n}-1)({n}-1)}
    \end{bmatrix} \in \mathbb{C}^{n\times n}
\end{equation*}
where $w=\exp{\left(-\frac{2 \imath\pi}{{n}}\right)}$  as a primitive $n$-th root of unity and $\imath =\sqrt{ -1}$. 
On the other hand, 
\begin{equation*}
    {\bcr}
     (\mc{T})  =
 (F_{n}^* \otimes I_{n})\cdot 
    \begin{pmatrix}
     {D}_1 &  &  & \\
     & {D}_2 &  & \\
     & &  \ddots &\\ 
    & &  & {D}_{n}
    \end{pmatrix}  \cdot (F_{n} \otimes I_{n}).
\end{equation*}

A tensor $\mc{T} \in \mathbb{C}^{m\times m\times n}$ is called an {\it identity tensor} if first frontal slice of $\mc{T}$ is an identity matrix of order $m$ and other frontal slices are zero matrices. The identity tensor $\mc{I} \in \mathbb{C}^{m\times m\times n}$ is denoted by $\mc{I}_{mmn}$. Equivalently, a tensor $\mc{T} \in \mathbb{C}^{m\times m\times n}$ is called  an {\it identity tensor} if $\bcr(\mc{T})$ is an identity matrix of order $mn$. For a tensor $\mc{T} \in \mathbb{R}^{m\times m\times n}$ there exist a tensor $\mc{X}$ such that 
  \begin{equation*}
  \mc{T} *\mc{X}=\mc{X}*\mc{T}=\mc{I}_{mmn},
  \end{equation*}
  here, $\mc{X}$ is the inverse of $\mc{T}$ and its denoted  by $\mc{T}^{-1}$. Following the $\bcr$ operation for an invertiable tensor $\mc{T}$, the above equality is equivalent to 
  \begin{equation*}
      \bcr(\mc{T})\cdot \bcr(\mc{T}^{-1}) = {I}_{np} = \bcr(\mc{T}) \cdot \bcr(\mc{T})^{-1}.  
  \end{equation*}
We subsequently, discuss the following useful properties of $\bcr$ and $\bcr^{-1}$ operation in the following lemma. 

\begin{lemma}
Let $\mc{S}, \mc{T} \in \mathbb{C}^{p\times q\times n}$. Then 
\begin{enumerate}
 \item[\rm (a)] $\bcr(\alpha_1\mc{S}+\alpha_2\mc{T}) = \alpha_1\bcr(\mc{S}) +\alpha_2\bcr(\mc{T})$, where $\alpha_1$ and $\alpha_2$ are scalars.
        \item[\rm (b)] $\bcr(\mc{S}*\mc{T}) = \bcr(\mc{S})* \bcr(\mc{T})$
         \item[\rm (c)] $(\bcr(\mc{S}))^{-1} = \bcr(\mc{S}^{-1})$
    \item[\rm (d)] $(\bcr(\mc{S}))^* = \bcr{(\mc{S}^*)}$, where $\mc{S}^*$ is the transpose conjugate of $\mc{S}$.
\end{enumerate}
\end{lemma}

\begin{lemma}
Let ${S}, {T} \in \mathbb{C}^{pn\times qn}$. Then 
\begin{enumerate}
 \item[\rm (a)] $\bcr^{-1}(\alpha_1S+\alpha_2T) = \alpha_1\bcr^{-1}({S}) +\alpha_2\bcr^{-1}({T})$,  where $\alpha_1$ and $\alpha_2$ are scalars. 
       \item[\rm (b)] $\bcr^{-1}({S}*{T}) = \bcr^{-1}({S})* \bcr^{-1}({T})$.
        \end{enumerate}
\end{lemma}

\begin{definition}[\cite{Kilmer13, miao2020t}]
Let  $\mc{T} \in \mathbb{R}^{p\times q\times n}$. Then $t$-range space and  $t$-null space and $t$-rank of $\mc{T}$ are defined below using $\bcr$ operation. 
 \begin{enumerate}
     \item[(a)] The $t$-range space of $\mc{T}$ is denoted by $\rg(\mc{T})$ and defined as $\rg(\mc{T})=R(\bcr(\mc{T}))$, where $R(T)$ stands for range space of the matrix $T$.
          \item[(b)] The $t$-null space of $\mc{T}$ is denoted by $\nl(\mc{T})$ and defined as $\nl(\mc{T})=N(\bcr(\mc{T}))$, where $N(T)$ stands for null space of the matrix $T$.
              \item[(c)] The norm of a tensor $\mc{T}$ is defined as $||\mc{T}||_2=||\bcr(\mc{T})||_2$ and 
              \begin{center}
                $\displaystyle ||\mc{T}||_F=\sqrt{\sum_{i=1}^p\sum_{j=1}^q\sum_{k=1}^n|a_{ijk}|^2}$.  
              \end{center}
              \item[(d)] The $t$-rank of a tensor $\mc{T}$ is denoted by $\rank(\mc{T})$ and defined as $\rank(\mc{T})=\text{rank}(\bcr(\mc{T}))$.               
              \item[(e)] The $t$-index of a tensor $\mc{T}$ is denoted by $ind_{t}(\mc{T})$ and defined as $ind_t(\mc{T})=ind(\bcr(\mc{T}))$, where $ind(T)$ is the smallest non negative integer $k$ satisfying $\text{rank}({T}^k)=\text{rank}({T}^{k+1})$. 
            \item[(e)] The condition number of a tensor $\mc{T}$ is defined as $k(\mc{T})=\|\mc{T}\|\cdot\|\mc{T}^{-1}\|$.
 \end{enumerate}
\end{definition}
\begin{table}[!htp]\label{Gtable}
\begin{center}
\caption{\small Equations required to define generalized inverses of tensors.}
\vspace{.3cm}
\begin{tabular}{|c|l|c|l|c|l|c|l|}
 \hline
Label & Equation  & Label & Equation & Label & Equation \\\hline 
$(1)$ & $\mc{T}*\mc{X}*\mc{T}=\mc{T}$  & $(2)$ & $\mc{X}*\mc{T}*\mc{X}=\mc{X}$ & $(4)$ & $(X*A)^*=X*A$ \\\hline
$\left(1^k\right)$ &   $\mc{T}^{k+1}*\mc{X}=\mc{T}^k$ & $(3)$ & $(\mc{T}*\mc{X})^*=\mc{T}*\mc{X}$ &$(5)$ &  $\mc{T}*\mc{X} = \mc{X}*\mc{T}$ \\\hline
\end{tabular}
\end{center}
\end{table}
Here we use the notation $\mc{T}^{(\lambda)}$ for an element of $\{\lambda \}$-inverse of $\mc{T}$ and $\mc{T}\{\lambda\}$ for the set of all $\{\lambda \}$-inverses of $\mc{T}$, where $\lambda\in \{1, 1^k, 2, 3, 4,5 \}$. For instance, a tensor $\mc{X} \in \mathbb{C}^{q\times p\times n} $ is called a  $\{1\}$-inverse of $\mc{T} \in \mathbb{C}^{p\times q\times n}$, if $X$ satisfies $(1)$, i.e., $\mc{T}*\mc{X}*\mc{T}=\mc{T}$ and we denote $X$ by $T^{(1)}$.  Further, if a tensor $\mc{Y}$ from the set $\mc{T}\{\lambda\}$ satisfying $\rg(\mc{X})=\rg(\mc{B})$ (resp. $\nl(\mc{X}) = \nl(\mc{C})$ is denoted by $\mc{T}^{(\lambda)}_{\rg(\mc{B}),*}$ (resp. $\mc{T}^{(\lambda)}_{*,\nl(\mc{C})}$). For $\mc{T} \in \mathbb{R}^{p\times q\times n}$ if  $\mc{X} \in \mc{T}\{1, 2,3,4\}$, then  $\mc{X}$ is called 
the Moore-Penrose inverse \cite{Miao2020generalized} of $\mc{T}$ and denoted by $\mc{T}^\dagger$. Similarly, if $ind(\mc{T}) = k$ and $\mc{X} \in \mc{T}\{1^k, 2,5\}$ , then  $\mc{X}$ is called  the Drazin inverse \cite{Cui2021perturbation,miao2020t} of $\mc{T}$. In particular, when $k=1$, we call $\mc{X}$ the Group inverse of $\mc{T}$ and denoted by $\mc{T}^{\#}$. The Drazin inverse is denoted by $\mc{T}^D$.  On the other hand, we can write $\{1\}$-inverse of $\mc{T} \in \mathbb{C}^{p\times q\times n}$ in the frame work of $\bcr$ and $\bcr^{-1}$ operations.

\begin{proposition}
Let $\mc{T} \in \mathbb{C}^{p\times q\times n}$. Then $\mc{T}^{(1)} = \bcr^{-1}[\bcr(\mc{T})^{(1)}]$.
\begin{proof} 
The proof follows from the below expression.
    \begin{eqnarray*}
&&    \mc{T}*\bcr^{-1}[\bcr(\mc{T})^{(1)}]*\mc{T} \\ & &\hspace{2.5cm}= \bcr^{-1}(\bcr(\mc{T}))* \bcr^{-1}[\bcr(\mc{T})^{(1)}]* \bcr^{-1}(\bcr(\mc{T})),\\ 
   & &\hspace{2.5cm}= \bcr^{-1}[\bcr(\mc{T})(\bcr(\mc{T}))^{(1)}\bcr(\mc{T})]\\
    & &\hspace{2.5cm}= \bcr^{-1}[\bcr(\mc{T})]= \mc{T}.
    \end{eqnarray*}
\end{proof}
\end{proposition}

\section{Main results}
We first prove the following auxiliary results, which are used to discuss the outer inverse of tensors based on $t$-product with a prescribed range and kernel of third-order tensors.
\begin{proposition}\label{rgprop}
    Let $\mc{S}\in \cpqn$, $\mc{T}\in\cpkn$ and $\mc{U}\in\ckqn$. Then 
    \begin{enumerate}
        \item[\rm (i)] $\rg(\mc{S})\subseteq \rg(\mc{T})$ if and only if $\mc{S}=\mc{T}*\mc{Z}$ for some $\mc{Z}\in\ckqn$.
\item[\rm (ii)] $\nl(\mc{S})\subseteq \nl(\mc{U})$ if and only if $\mc{U}=\mc{Y}*\mc{S}$ for some $\mc{Y}\in\mathbb{C}^{k\times p\times n}$.
    \end{enumerate}
\end{proposition}
\begin{proof}
(i) If $\mc{S}=\mc{T}*\mc{Z}$ for some $\mc{Z}\in\ckqn$, then $\bcr(\mc{S})=\bcr(\mc{T}*\mc{Z})=\bcr (\mc{T})\cdot\bcr (\mc{Z})$. We can verify that 
\[\rg(\mc{S})=R(\bcr(\mc{S}))\subseteq R(\bcr(\mc{T}))=\rg(\mc{T}).\]
Conversely, let $\rg(\mc{S})\subseteq \rg(\mc{T})$. Then $R(\bcr(\mc{S}))\subseteq R(\bcr(\mc{T}))$ and subsequently we have $\bcr(\mc{S})=\bcr(\mc{T}) Z=\bcr(\mc{T})\bcr(\mc{Z})$, where $Z=\bcr(\mc{Z})$ and $Z\in \mathbb{C}^{kn\times qn}$. Thus $\mc{S}=\bcr^{-1}\bcr(\mc{S})=\bcr^{-1}\left(\bcr(\mc{T})\bcr(\mc{Z})\right)=\mc{T}*\mc{Z}$ for some $\mc{Z}\in\ckqn$.\\
(ii) The proof will be similar to part (i). 
\end{proof}
In view of Proposition \ref{rgprop} (i), we can show the below results. 
\begin{corollary}\label{lcan}
    Let $\mc{S},\mc{T}\in \cpqn$, $\mc{U}\in\mathbb{C}^{s\times k \times n}$, and $\mc{V}\in\mathbb{C}^{k\times p \times n}$. If $\mc{U}*\mc{V}*\mc{S}=\mc{U}*\mc{V}*\mc{T}$ and $\rg(\mc{V}^{\top})=\rg((\mc{U}*
    \mc{V})^{\top})$ then $\mc{V}*\mc{S}=\mc{V}*\mc{T}$.
\end{corollary}
\begin{corollary}\label{rcan}
     Let $\mc{S},\mc{T}\in \cpqn$, $\mc{U}\in\cqkn$, and $\mc{V}\in\mathbb{C}^{k\times s \times n}$. If $\mc{S}*\mc{U}*\mc{V}=\mc{T}*\mc{U}*\mc{V}$ and $\rg(\mc{U})=\rg(\mc{U}*
    \mc{V})$ then $\mc{S}*\mc{U}=\mc{T}*\mc{U}$.
\end{corollary}
\begin{lemma}\label{lemout}
    Let $\mc{S}\in \cpqn$ and $\mc{T}\in\cqkn$. Then the following statements are equivalent:
\begin{enumerate}
    \item[\rm (i)] there exists a tensor $\mc{X}$ such that $\mc{S}^{(2)}_{\rg(\mc{T}),*}=\mc{X}$.
    \item[\rm (ii)] $\mc{T}=\mc{T}*\mc{Z}*\mc{S}*\mc{T}$ for some $\mc{Z}\in \ckpn$.
    \item[\rm (iii)] $\nl(\mc{T})=\nl(\mc{S}*\mc{T})$.
    \item[\rm (iv)] $\mc{T}=\mc{T}*(\mc{S}*\mc{T})^{(1)}*\mc{S}*\mc{T}$.
    \item[\rm (v)] $\rank(\mc{S}*\mc{T})={\rank}(\mc{T})$.
\end{enumerate}
\end{lemma}
\begin{proof}
(i)$\Rightarrow$(ii): Let $\mc{X}=\mc{S}^{(2)}_{\rg(\mc{T}),*}$. Then $\mc{X}*\mc{S}*\mc{X}=\mc{X}$ and $\rg(\mc{X})=\rg(\mc{T})$. Applying Proposition \ref{rgprop} (i), we obtain \[\mc{T}=\mc{X}*\mc{Z}_1=\mc{X}*\mc{S}*\mc{X}*\mc{Z}_1=\mc{X}*\mc{S}*\mc{T}=\mc{T}*\mc{Z}*\mc{S}*\mc{T}\mbox{ for some $\mc{Z}\in\ckpn$.}\] 
(ii)$\Rightarrow$(iii): Let $\mc{T}=\mc{T}*\mc{Z}*\mc{S}*\mc{T}$  some $\mc{Z}\in \ckpn$. In view of Proposition \ref{rgprop} (ii), we have $\nl(\mc{S}*\mc{T})\subseteq \nl(\mc{T})$. The other inclusion $\nl(\mc{T})\subseteq \nl(\mc{S}*\mc{T})$ follows trivially by Proposition \ref{rgprop} (ii).\\
(iii)$\Rightarrow$(iv): Let $\nl(\mc{T}=\nl(\mc{S}*\mc{T})$. Then there exist $\mc{Z}\in $ such that $\mc{T}=\mc{Z}*\mc{S}*\mc{T}$. Further, 
\[\mc{T}=\mc{Z}*\mc{S}*\mc{T}=\mc{Z}*\mc{S}*\mc{T}*(\mc{S}*\mc{T})^{(1)}*\mc{S}*\mc{T}=\mc{T}*(\mc{S}*\mc{T})^{(1)}*\mc{S}*\mc{T}.\]
(iv)$\Rightarrow$(i): Define $\mc{X}=\mc{T}*(\mc{S}*\mc{T})^{(1)}$ and let $\mc{T}=\mc{T}*(\mc{S}*\mc{T})^{(1)}*\mc{S}*\mc{T}=\mc{X}*\mc{S}*\mc{T}$. Then it is easy to verify that 
$\rg(\mc{X})\subseteq \rg(\mc{T})$,  $\rg(\mc{T})\subseteq \rg(\mc{X})$, and 
\[\mc{X}*\mc{S}*\mc{X}=\mc{T}*(\mc{S}*\mc{T})^{(1)}*\mc{S}*\mc{T}*(\mc{S}*\mc{T})^{(1)}=\mc{T}*(\mc{S}*\mc{T})^{(1)}=\mc{X}.\]
(iii)$\Leftrightarrow$(v): It follows from the equivalence between $N(\bcr(\mc{T}))=N(\bcr(\mc{S}*\mc{T}))$ and $\text{rank}(\bcr(\mc{T}))=\text{rank}(\bcr(\mc{S}*\mc{T}))$.
\end{proof}
\begin{theorem}\label{thmrb}
Assume that the statements in Lemma \ref{lemout} are true, then $\mc{S}\{2\}_{\rg(\mc{T}),*}$ will have the following representations:
   \begin{enumerate}
       \item[\rm (i)] $\mc{S}\{2\}_{\rg(\mc{T}),*}=\mc{T}*\mc{Y}=\mc{T}*\mc{Z}$, where $\mc{Y}\in (\mc{S}*\mc{T})\{1\}$ and $\mc{Z}$ satisfying $\mc{T}*\mc{Z}*\mc{S}*\mc{T}=\mc{T}$.
       \item[\rm (ii)] $\mc{S}\{2\}_{\rg(\mc{T}),*}=\mc{T}*(\mc{S}*\mc{T})^{(1)}+\mc{T}*\mc{Z}-\mc{T}*\mc{Z}*\mc{S}*\mc{T}*(\mc{S}*\mc{T})^{(1)},$ where $\mc{Z}$ is arbitrary.
   \end{enumerate}
   \end{theorem}
   \begin{proof}
       (i) Let $\mc{X}\in \mc{S}\{2\}_{\rg(\mc{T}),*}$. Then $\rg(\mc{X})=\rg(\mc{T})$. Further, $\mc{X}=\mc{T}*\mc{Z}$ for some $\mc{Z}\in\ckpn$. Now $\mc{T}*\mc{Z}=\mc{X}=\mc{X}*\mc{S}*\mc{X}=\mc{T}*\mc{Z}*\mc{S}*\mc{T}*\mc{Z}$. By applying Corollary \ref{rcan}, we obtain $\mc{T}=\mc{T}*\mc{Z}*\mc{S}*\mc{T}$. Hence we conclude that 
       \begin{equation}\label{eqotrg}
          \mc{S}\{2\}_{\rg(\mc{T}),*}\subseteq  \{\mc{T}*\mc{Z}\,|\,\mc{Z}\in\ckpn, \mc{T}*\mc{Z}*\mc{S}*\mc{T}=\mc{T}\}.
       \end{equation}
    From $\mc{T}*\mc{Z}*\mc{S}*\mc{T}=\mc{T}$, we have $\mc{Z}\in (*\mc{S}*\mc{T})^{(1)}$ and hence 
    \begin{equation}\label{eqotrg1}
        \{\mc{T}*\mc{Z}\,|\,\mc{Z}\in\ckpn, \mc{T}*\mc{Z}*\mc{S}*\mc{T}=\mc{T}\}\subseteq \{\mc{T}*\mc{Y}\,|\,\mc{Y}\in (\mc{S}*\mc{T})^{(1)}\}.
    \end{equation}
    By Lemma \ref{lemout}, we get 
    \begin{equation}\label{eqotrg2}
    \{\mc{T}*\mc{Y}\,|\,\mc{Y}\in (\mc{S}*\mc{T})^{(1)}\}\subseteq \mc{S}\{2\}_{\rg(\mc{T}),*}.
\end{equation}
 The proof is complete by virtue of the equations \eqref{eqotrg}, \eqref{eqotrg1}, and \ref{eqotrg2}. \\
(ii) Let $\mc{T}*\mc{Z}*\mc{S}*\mc{T}=\mc{T}$. Then
\begin{equation}\label{bsys}
 \bcr(\mc{T})\bcr(\mc{Z})\bcr(\mc{S})\bcr(\mc{T})=\bcr(\mc{T})   
\end{equation}

From Lemma \ref{lemout} (v), we have  $\text{rank}((\bcr(\mc{S})\bcr(\mc{T}))^{\top})=\text{rank}((\bcr(\mc{T})^{\top})$ and by Corollary \ref{lcan} the expression 
\[\bcr(\mc{S})\bcr(\mc{T})(\bcr(\mc{S})\bcr(\mc{T}))^{(1)}\bcr(\mc{S})\bcr(\mc{T})=\bcr(\mc{S})\bcr(\mc{T})\] yields $\bcr(\mc{T})(\bcr(\mc{S})\bcr(\mc{T}))^{(1)}\bcr(\mc{S})\bcr(\mc{T})=\bcr(\mc{T})$. Thus the system \eqref{bsys} is consistent, and the general solution is given by 
\begin{eqnarray*}
\bcr(\mc{Z})&=&(\bcr(\mc{T}))^{(1)}*\bcr(\mc{T})*(\bcr(\mc{S})\bcr(\mc{T}))^{(1)}+\bcr(\mc{Z})\\
  && \hspace*{-2.5cm} -(\bcr(\mc{T}))^{(1)}*\bcr(\mc{T})*\bcr(\mc{Z})*\bcr(\mc{S})\bcr(\mc{T})*(\bcr(\mc{S})\bcr(\mc{T}))^{(1)}.
\end{eqnarray*}
Applying $\bcr^{-1}$ both sides, we obtain
\[\mc{Z}=\mc{T}^{(1)}*\mc{T}*(\mc{S}*\mc{T})^{(1)}+\mc{Z}-\mc{T}^{(1)}*\mc{T}*\mc{Z}*\mc{S}*\mc{T}*(\mc{S}*\mc{T})^{(1)},\mbox{ and}\]
\begin{equation}\label{eqrb}
    \mc{T}*\mc{Z}=\mc{T}*(\mc{S}*\mc{T})^{(1)}+\mc{T}*\mc{Z}-\mc{T}*\mc{Z}*\mc{S}*\mc{T}*(\mc{S}*\mc{T})^{(1)}.
\end{equation}
The proof follows from equation \eqref{eqrb} and part (i).
\end{proof}
In view of Lemma \ref{lemout} and Theorem \ref{thmrb}, the following results can be proved similarly.
\begin{lemma}\label{lemout1}
Let $\mc{S}\in \cpqn$ and $\mc{T}\in\ckpn$. Then the following statements are equivalent:
\begin{enumerate}
    \item[\rm (i)] there exists a tensor $\mc{X}$ such that $\mc{S}^{(2)}_{*,\nl(\mc{T})}=\mc{X}$.
    \item[\rm(ii)] $\mc{T}=\mc{T}*\mc{S}*\mc{Z}*\mc{T}$ for some $\mc{Z}\in \mathbb{C}^{q \times k\times n}$.
    \item[\rm(iii)] $\rg(\mc{T}*\mc{S})=\rg(\mc{T})$.
    \item[\rm(iv)] $\mc{T}=\mc{T}*\mc{S}*(\mc{T}*\mc{S})^{(1)}*\mc{T}$.
    \item[\rm(v)] $\rank(\mc{T}*\mc{S})={\rank}(\mc{T})$.
\end{enumerate}
\end{lemma}

\begin{algorithm}[hbt!]
  \caption{Computation of outer inverse $\mc{X} = \mc{S}^{(2)}_{\rg(\mc{T}), *}$} \label{AlgoTran}
  \begin{algorithmic}[1]
    \Procedure{oinverse}{$\mc{S}$}
    \State \textbf{Input} $\mc{S}\in \cpqn$ and $\mc{T}\in \cqkn$.
      \For{$i \gets 1$ to $n$} 
      \State $\mc{S}=\textup{fft}(\mc{S}, [~ ], i);$ ~ $\mc{T}=\textup{fft}(\mc{T}, [~ ], i);$
      \EndFor
      \For{$i \gets 1$ to $n$} 
       \State $\mc{Z}(:, :, i ) = \mc{S}(:, :, i )\mc{T}(:, :, i );$
      \EndFor
      \For{$i \gets n$ to $1$} 
      \State $\mc{Z} \leftarrow \textup{ifft}(\mc{Z}, [~ ], i);$
      \EndFor
     \If{$\rank{(\mc{Z})} = \rank{(\mc{T})}$}
      \State \textbf{compute} $\mc{Z}^{(1)}$ using Algorithm 2 \cite{RatiJMZ} 
    \EndIf
    \State \textbf{compute} $\mc{T}*\mc{Z}^{(1)}$  
      \State \textbf{return } $\mc{S}^{(2)}_{\rg(\mc{T}), *} = \mc{T}*\mc{Z}^{(1)}$ 
    \EndProcedure
  \end{algorithmic}
\end{algorithm}

\begin{example}\label{Ex1}
Let $ \mathcal{S} \in \mathbb{R}^{2 \times 2 \times 3}$  and $ \mathcal{T} \in \mathbb{R}^{2 \times 3 \times 3}$  
 with frontal slices 
\begin{eqnarray*}
\mc{S}_{1} =
    \begin{pmatrix}
    1 & 1 \\
    -2 & 0
    \end{pmatrix},~
\mc{S}_{2} =
    \begin{pmatrix}
     0 & 1\\
     1 & -2
    \end{pmatrix},~~
    \mc{S}_{3} =
    \begin{pmatrix}
     0 & -1 \\
     1 & 2
    \end{pmatrix},
    \end{eqnarray*}
    \begin{eqnarray*}
\mc{T}_{1} =
    \begin{pmatrix}
    -1 & 1 & -2\\
    -2 & 1 & -2
    \end{pmatrix},~
\mc{T}_{2} =
    \begin{pmatrix}
     -2 & 1 & 1\\
     2 & -2 & 0
    \end{pmatrix},~~
    \mc{T}_{3} =
    \begin{pmatrix}
     2 & -1 & 2\\
     0 & 1 & 2
    \end{pmatrix}.
    \end{eqnarray*}
Then we can verify that  $\rank{(\mc{S}*\mc{T})} = \rank{(\mc{T})} = 5$. With the help of $(\mc{S}*\mc{T})^{(1)} \in \mathbb{R}^{3\times 2 \times 3}$ (Algorithm 2 \cite{RatiJMZ}) we can compute $\mc{X}=\mc{S}^{(2)}_{\rg(\mc{T}), *}$, where  
\begin{eqnarray*}
\mc{X}_{1} =
    \begin{pmatrix}
    0 & -1/3 \\
    1/2 & 1/6
    \end{pmatrix},~
\mc{X}_{2} =
    \begin{pmatrix}
     0 & 0\\
     -1/2 & -1/6
    \end{pmatrix},~~
 \mc{X}_{3} =
    \begin{pmatrix}
     1 & 1/3 \\
     0 & 0
    \end{pmatrix}.
    \end{eqnarray*}
\end{example}

\begin{theorem}\label{thmrc}
Assume that the statements in Lemma \ref{lemout1} are true, then  $\mc{S}\{2\}_{*,\nl(\mc{T})}$ will have the following representations:
   \begin{enumerate}
       \item[\rm (i)] $\mc{S}\{2\}_{*,\nl(\mc{T})}=\mc{Y}*\mc{T}=\mc{Z}*\mc{T}$, where $\mc{Y}\in (\mc{T}*\mc{S})\{1\}$ and $\mc{Z}$ satisfying $\mc{T}*\mc{S}*\mc{Z}*\mc{T}=\mc{T}$.
       \item[\rm (ii)] $\mc{S}\{2\}_{*,\nl(\mc{T})}=(\mc{T}*\mc{S})^{(1)}*\mc{T}+\mc{Z}*\mc{T}-(\mc{T}*\mc{S})^{(1)}*\mc{T}*\mc{S}*\mc{Z}*\mc{T}$, where $\mc{Z}$ is arbitrary.
   \end{enumerate}
   \end{theorem}
\begin{lemma}\label{lembc}
    Let $\mc{T} \in \mathbb{C}^{p\times q\times n}$, $\mc{B} \in \mathbb{C}^{q\times k\times n}$,  and $\mc{C} \in \mathbb{C}^{s\times p\times n}$. Then the following statements are equivalent:
    \begin{enumerate}
        \item[\rm (i)] there exists a tensor $\mc{X}$ such that $\mc{T}^{(2)}_{\rg(\mc{B}),\nl(\mc{C})}=\mc{X}$. 
        \item[\rm (ii)] there exist tensors $\mc{U},\mc{V} \in \mathbb{C}^{k\times s\times n}$ such that
        \[ \mc{B}*\mc{U}*\mc{C}*\mc{T}*\mc{B}=\mc{B},~ \mbox{and } \mc{C}*\mc{T}*\mc{B}*\mc{V}*\mc{C}=\mc{C}.\]
        \item[\rm (iii)] there exist tensors $\mc{Y}\in \mathbb{C}^{k\times p\times n}$ and $\mc{Z}\in \mathbb{C}^{q\times s\times n}$ such that 
         \[ \mc{B}*\mc{Y}*\mc{T}*\mc{B}=\mc{B},~  \mc{C}*\mc{T}*\mc{Z}*\mc{C}=\mc{C},~\mbox{and }\mc{B}*\mc{Y}=\mc{Z}*\mc{C}.\]
        \item[\rm (iv)] there exist tensors $\mc{U}\in \mathbb{C}^{k\times p\times n}$ and $\mc{V}\in \mathbb{C}^{q\times s\times n}$ such that 
         \[ \mc{C}*\mc{T}*\mc{B}*\mc{U}=\mc{C},~\mbox{and }  \mc{V}*\mc{C}*\mc{T}*\mc{B}=\mc{B}.\]
        \item[\rm (v)] $\rg(\mc{C}*\mc{T}*\mc{B})=\rg(\mc{C})$ and $\nl(\mc{C}*\mc{T}*\mc{B})=\nl(\mc{B})$.
        \item[\rm (vi)] $\mc{B}*(\mc{C}*\mc{T}*\mc{B})^{(1)}*\mc{C}*\mc{T}*\mc{B}=\mc{B}$ and $\mc{C}*\mc{T}*\mc{B}*(\mc{C}*\mc{T}*\mc{B})^{(1)}*\mc{C}=\mc{C}$.
        \item[\rm (vii)] $\rank(\mc{C}*\mc{T}*\mc{B})=\rank(\mc{B})=\rank(\mc{C})$.
            \end{enumerate}
\end{lemma}
\begin{proof}
It follows from Lemma \ref{lemout} and Lemma \ref{lemout1}.
\end{proof}
The following theorem is an immediate consequence of Lemma \ref{lembc}.

\begin{algorithm}[hbt!]
  \caption{Computation of outer inverse $\mc{X} = \mc{S}^{(2)}_{*, \nl(\mc{T})}$} \label{Algonull}
  \begin{algorithmic}[1]
    \Procedure{oinverse}{$\mc{S}$}
    \State \textbf{Input}  $\mc{S}\in \cpqn$ and $\mc{T}\in \ckpn$.
      \For{$i \gets 1$ to $n$} 
      \State $\mc{S}=\textup{fft}(\mc{S}, [~ ], i);$ ~ $\mc{T}=\textup{fft}(\mc{S}, [~ ], i);$
      \EndFor
      \For{$i \gets 1$ to $n$} 
       \State $\mc{Z}(:, :, i ) = \mc{T}(:, :, i )\mc{S}(:, :, i );$
      \EndFor
      \For{$i \gets n$ to $1$} 
      \State $\mc{Z} \leftarrow \textup{ifft}(\mc{Z}, [~ ], i);$
      \EndFor
     \If{$\rank{(\mc{Z})} = \rank{(\mc{T})}$}
      \State \textbf{compute} $\mc{Z}^{(1)}$ using Algorithm 2 \cite{RatiJMZ} 
    \EndIf
    \State \textbf{compute} $\mc{Z}^{(1)}*\mc{T}$  
      \State \textbf{return } $\mc{S}^{(2)}_{*, \nl(\mc{T})} = \mc{Z}^{(1)}*\mc{T}$ 
    \EndProcedure
  \end{algorithmic}
\end{algorithm}

\begin{example} 
Consider the tensor $ \mathcal{S}$ defined in Example \ref{Ex1} and $ \mathcal{T} \in \mathbb{R}^{3 \times 2 \times 3}$  with frontal slices 
 \begin{eqnarray*}
\mc{T}_{1} =
    \begin{pmatrix}
    0 & 1 \\
    1 & -1 \\
    0 & 1
    \end{pmatrix},~
\mc{T}_{2} =
    \begin{pmatrix}
     1 & 0 \\
     0 & 0 \\
     1 & 0
    \end{pmatrix},~~
    \mc{T}_{3} =
    \begin{pmatrix}
     0 & 0 \\
     -1 & 1 \\
     1 & 1
    \end{pmatrix}.
    \end{eqnarray*}
Then we can verify that  $\rank{(\mc{T}*\mc{S})} = \rank{(\mc{T})} = 5$. With the help of $(\mc{T}*\mc{S})^{(1)} \in \mathbb{R}^{2\times 3 \times 3}$ (Algorithm 2 \cite{RatiJMZ}) we can compute $\mc{X}=\mc{S}^{(2)}_{*,\nl(\mc{T})} = (\mc{T}*\mc{S})^{(1)}*\mc{T}$, where  
\begin{eqnarray*}
\mc{X}_{1} =
    \begin{pmatrix}
    -1/6 & -1/6 \\
    2/3 & 1/3
    \end{pmatrix},~
\mc{X}_{2} =
    \begin{pmatrix}
    -1/6           & 1/6  \\   
      -1/3          & 0
    \end{pmatrix},~~
 \mc{X}_{3} =
    \begin{pmatrix}
     5/6     &       1/2   \\  
       1/6    &        1/6 
    \end{pmatrix}.
    \end{eqnarray*}
\end{example}

\begin{theorem}
   If the statements in Lemma \ref{lembc} are satisfied, then 
   \[\mc{T}^{(2)}_{\rg(\mc{B}),\nl(\mc{C})}=\mc{B}*(\mc{C}*\mc{T}*\mc{B})^{(1)}*\mc{C}=\mc{B}*\mc{Z}*\mc{C},\]
   where $\mc{Z}$ satisfies $\mc{B}*\mc{Z}*\mc{C}*\mc{T}*\mc{B}=\mc{B}$ and $\mc{C}*\mc{T}*\mc{B}*\mc{Z}*\mc{C}=\mc{C}$. 
\end{theorem}

\begin{algorithm}[hbt!]
  \caption{Computation of outer inverse $\mc{X} = \mc{S}^{(2)}_{\rg({B}), \nl({C})}$} \label{Algoouter}
  \begin{algorithmic}[1]
    \Procedure{oinverse}{$\mc{S}$}
    \State \textbf{Input} $\mc{T} \in \mathbb{C}^{p\times q\times n}$, $\mc{B} \in \mathbb{C}^{q\times k\times n}$,  and $\mc{C} \in \mathbb{C}^{s\times p\times n}$. 
      \For{$i \gets 1$ to $n$} 
      \State $\mc{T}=\textup{fft}(\mc{T}, [~ ], i);$ ~ $\mc{B}=\textup{fft}(\mc{B}, [~ ], i);$~ $\mc{C}=\textup{fft}(\mc{C}, [~ ], i);$
      \EndFor
      \For{$i \gets 1$ to $n$} 
       \State $\mc{Z}(:, :, i ) = \mc{C}(:, :, i )\mc{T}(:, :, i )\mc{B}(:, :, i );$
      \EndFor
      \For{$i \gets n$ to $1$} 
      \State $\mc{Z} \leftarrow \textup{ifft}(\mc{Z}, [~ ], i);$
      \EndFor
     \If{$\rank{(\mc{Z})} = \rank{(\mc{B})} =  \rank{(\mc{C})}$}
      \State \textbf{compute} $\mc{Z}^{(1)}$ using Algorithm 2 \cite{RatiJMZ} 
    \EndIf
    \State \textbf{compute} $\mc{B}*\mc{Z}^{(1)}*\mc{C}$  
      \State \textbf{return } $\mc{S}^{(2)}_{\rg({B}), \nl({C})} = \mc{B}*\mc{Z}^{(1)}*\mc{C}$ 
    \EndProcedure
  \end{algorithmic}
\end{algorithm}

\begin{example} 
Consider the tensor $ \mathcal{T} = \mathcal{S}$, where $\mathcal{S}$ defined in Example \ref{Ex1}. Let $ \mathcal{B} \in \mathbb{R}^{2 \times 3 \times 3}$ and $ \mathcal{C} \in \mathbb{R}^{3 \times 2 \times 3}$  with frontal slices 
 \begin{eqnarray*}
\mc{B}_{1} =
    \begin{pmatrix}
    1 & 2 & 1\\
    0 & 0 & 1
    \end{pmatrix},~
\mc{B}_{2} =
    \begin{pmatrix}
     1 & 2 & 1\\
    0 & 0 & 1
    \end{pmatrix},~~
    \mc{B}_{3} =
    \begin{pmatrix}
     1 & 2 & 1\\
    0 & 0 & 1
    \end{pmatrix},
    \end{eqnarray*}
    \begin{eqnarray*}
\mc{C}_{1} =
    \begin{pmatrix}
    1 & 2 \\
    0 & 0 \\
    1 & 1
    \end{pmatrix},~
\mc{C}_{2} =
    \begin{pmatrix}
     1 & 2 \\
     1 & 0 \\
     1 & 1
    \end{pmatrix},~~
    \mc{C}_{3} =
    \begin{pmatrix}
     1 & 2 \\
     1 & 0 \\
     1 & 1
    \end{pmatrix}.
    \end{eqnarray*}
Then we can verify that  $\rank{(\mc{C}*\mc{T}*\mc{B})} = \rank{(\mc{B})} =\rank{(\mc{C})} = 24$, Further, using Algorithm 2 \cite{RatiJMZ} we can compute $(\mc{C}*\mc{T}*\mc{B})^{(1)}$
\begin{eqnarray*}
(\mc{C}*\mc{T}*\mc{B})^{(1)}_{1} =
    \begin{pmatrix}
     1 & 2 & 1\\
    0 & 0 & 1\\
    0 & 0 & 1
    \end{pmatrix},~
(\mc{C}*\mc{T}*\mc{B})^{(1)}_{2} =
    \begin{pmatrix}
     1 & 2 & 1\\
    0 & 0 & 1\\
    0 & 0 & 1
    \end{pmatrix},~~
 (\mc{C}*\mc{T}*\mc{B})^{(1)}_{3} =
    \begin{pmatrix}
     1 & 2 & 1\\
    0 & 0 & 1\\
    0 & 0 & 1
    \end{pmatrix}.
    \end{eqnarray*}
Following $(\mc{C}*\mc{T}*\mc{B})^{(1)}$ we can compute $\mc{X}=\mc{T}^{(2)}_{\rg(\mc{B}),\nl(\mc{C})} = \mc{B}*(\mc{C}*\mc{T}*\mc{B})^{(1)}*\mc{C}$, where  
\begin{eqnarray*}
\mc{X}_{1} =
    \begin{pmatrix}
    1 & 2 \\
    0 & 0
    \end{pmatrix},~
\mc{X}_{2} =
    \begin{pmatrix}
     1 & 2\\
     1 & 0
    \end{pmatrix},~~
 \mc{X}_{3} =
    \begin{pmatrix}
     1 & 2 \\
     1 & 0 
    \end{pmatrix}.
    \end{eqnarray*}
\end{example}

The other generalized inverses can be derived as a special case from the outer generalized inverse with a prescribed range and null space, as given in the result below.

\begin{theorem}
    Let $\mc{S}\in \cpqn$ and $\mc{T}\in \mathbb{C}^{p\times p\times n}$ with $ind(\mc{T})=k$. Then 
    \begin{enumerate}
        \item[\rm (i)] $\mc{S}^{\dagger}=\mc{S}^{(2)}_{\rg(\mc{S}^*),\nl(\mc{S}^*)}$.
        \item[\rm (ii)] $\mc{T}^{D}=\mc{T}^{(2)}_{\rg(\mc{T}^k),\nl(\mc{T}^k)}$.
    \end{enumerate}
\end{theorem}
\begin{proof}
 (i) It is sufficient to  show $\rg(\mc{S}^{\dagger})=\rg(\mc{S}^{*})$ and $\nl(\mc{S}^{\dagger})=\nl(\mc{S}^{\dagger})$. From the below identities:
 \[\bcr(\mc{S}^*)=\bcr(\mc{S}^\dagger)\bcr(\mc{S})\bcr(\mc{S}^*)\mbox{ and  }\bcr(\mc{S}^\dagger)=\bcr(\mc{S}^*)\bcr((\mc{A}^\dagger)^*)\bcr(\mc{A}^{\dagger}),\]
 we obtain $R(\bcr(\mc{S}^*)=R(\bcr(\mc{S}^\dagger))$. Thus $\rg(\mc{S}^{\dagger})=\rg(\mc{S}^{*})$. If $Z\in \nl(\mc{S}^\dagger)$ then $\bcr(\mc{S}^\dagger) Z=O$. Now $\bcr(\mc{S}^*)Z=\bcr(\mc{S}^*)\bcr(\mc{S})\bcr(\mc{S}^\dagger) Z=O$ and hence $Z\in N(\bcr(\mc{S}^*))=\nl(\mc{S^*})$. Conversely, if $Z\in \nl(\mc{S}^*)=N(\bcr(\mc{S}^*))$ then $\bcr(\mc{S}^*)Z=O$. Moreover,
 \[\bcr(\mc{S}^\dagger)Z=\bcr(\mc{S}^\dagger)(\bcr(\mc{S}^\dagger)^*\bcr(\mc{S}^*)Z=O.\]
 Hence completes the proof of (i).\\
 (ii) From $\mc{T}^D=\mc{T}^k*(\mc{T}^D)^{k+1}$,  $\mc{T}^k=\mc{T}^D*\mc{T}^k$, and applying Lemma \ref{rgprop} (i), we get $\rg(\mc{T}^D)=\rg(\mc{T}^k)$.  If $Z\in \nl(\mc{T}^D)=N(\bcr(\mc{T}^D))$ then $\bcr(\mc{T}^D) Z=O$. Now $\bcr(\mc{T}^k)Z=\bcr(\mc{T}^{k+1})\bcr(\mc{T}^D)=O$ and hence $Z\in N(\bcr(\mc{T}^k))=\nl(\mc{T}^k)$. Conversely, if $Z\in \nl(\mc{T}^k)=N(\bcr(\mc{T}^k))$ then $\bcr(\mc{T}^k)Z=O$. In addition, $\bcr(\mc{T}^D)Z=(\bcr(\mc{T}^D)^{k+1}\bcr(\mc{T}^k)=O$. Thus $N(\bcr(\mc{T}^D))=N(\bcr(\mc{T}^k))$ and completes the proof.
 \end{proof}
\begin{lemma}\label{pfullrank}
Let $ \mc{S}\in\mathbb{C}^{q\times p\times n},\mc{T}\in\cpqn$ with $rank(\mc{T})=rn$. If the following are satisfied:
\begin{enumerate}
    \item[\rm (a)] $\rg(\mc{S})$  is a subspace of $\mathbb{C}^{pn}$ with dimension $sn\leq rn$.
    \item[\rm (b)] $\nl(\mc{S})$ is a subspace of $\mathbb{C}^{qn}$ with dimension $qn-sn$.
    \item[\rm (c)] $\mc{S}=\mc{B}*\mc{C}$ is a full-rank decomposition of $\mc{S}$.
    \item[\rm (d)] $\mc{T}^{(2)}_{\rg(\mc{S}),\nl(\mc{S})}$ exists.
\end{enumerate}
Then
\begin{enumerate}
    \item[\rm (i)] $\mc{C}*\mc{T}*\mc{B}$ is invertible. 
    \item[\rm (ii)] $\mc{T}^{(2)}_{\rg(\mc{S}),\nl(\mc{S})}=\mc{B}(\mc{C}*\mc{T}*\mc{B})^{-1}*\mc{C}=\mc{T}^{(2)}_{\rg(\mc{B}),\nl(\mc{C})}$.
\end{enumerate}
\end{lemma}
\begin{proof}
Let $S=\bcr(\mc{S})$, $T=\bcr(\mc{T})$, $B=\bcr(\mc{B})$ and $\bcr(\mc{C})$. From the given assumption, we have 
\begin{itemize}
    \item $R(\bcr(\mc{S}))$  is a subspace of $\mathbb{C}^{pn}$ with dimension $sn\leq rn$.
    \item $N(\bcr(\mc{S
    }))$ is a subspace of $\mathbb{C}^{qn}$ with dimension $qn-sn$.
    \item $\bcr(\mc{S})=\bcr(\mc{B})*\bcr(\mc{C})$ is a full-rank decomposition of $\bcr(\mc{S})$.
    \item $\bcr(\mc{T})^{(2)}_{\rg(\bcr(\mc{S})),\nl(\bcr(\mc{S}))}$ exists.
\end{itemize}
By using Proposition 2.1 \cite{peduotmat}, we obtain
$\bcr(\mc{C})\bcr(\mc{T})\bcr(\mc{B})$ invertible and 
\begin{eqnarray*}\label{fl-rank}
\bcr(\mc{T})^{(2)}_{\rg(\bcr(\mc{S})),\nl(\bcr(\mc{S}))}&=&\bcr(\mc{B})(\bcr(\mc{C})\bcr(\mc{T})\bcr(\mc{B}))^{-1}\bcr(\mc{C})\\
&=&\bcr(\mc{T})^{(2)}_{\rg(\bcr(\mc{B})),\nl(\bcr(\mc{C}))}. 
\end{eqnarray*}
Applying the inverse operation $\bcr^{-1}$, we can get the invertibility of $\mc{C}*\mc{T}*\mc{B}$ and  $\mc{T}^{(2)}_{\rg(\mc{S}),\nl(\mc{S})}=\mc{B}(\mc{C}*\mc{T}*\mc{B})^{-1}*\mc{C}=\mc{T}^{(2)}_{\rg(\mc{B}),\nl(\mc{C})}$.
\end{proof}

Chan \cite{Chan1987rank} presented an algorithm for computing a column permutation $P$ and a QR factorization $AP = QR$ of $m > n$ matrix $A$ such that a possible rank 
deficiency of $A$ will be revealed in the triangular factor $R$. Next we extend it to the tensor case.

\begin{theorem}\label{qrthm}
 Let $ \mc{S}\in\cpqn$, $\mc{T}\in\mathbb{C}^{q\times p\times n}$,  $\rank(\mc{T})=sn$,  and $\rank(\mc{S})=rn$ with $sn\leq rn$. Consider the $\mc{Q}*\mc{R}$  decomposition of $\mc{T}$ is of the form
 \begin{equation}
     \mc{T}*\mc{P}=\mc{Q}*\mc{R},
 \end{equation}
 where $\mc{P}\in\mathbb{C}^{p\times p\times n}$ is a permutation tensor, $\mc{Q}\in \mathbb{C}^{q\times q\times n}$ satisfying $\mc{Q}*\mc{Q}=\mc{I}_{qqn}$, and $\mc{R}\in\mathbb{C}^{q\times p\times n}$ with $\rank(\mc{R})=sn$. The tensor $\mc{P}$ to be chosen so that it partitions $\mc{Q}$ and $\mc{R}$ in the following form
 \begin{equation}
     \mc{Q}=\begin{bmatrix}
      \tilde{\mc{Q}} & \mc{Q}_2   
     \end{bmatrix},~\mc{R}=\begin{bmatrix}
         \mc{R}_{11} & \mc{R}_{22}\\
         \mc{O} &\mc{O}
     \end{bmatrix}=\begin{bmatrix}
         \tilde{\mc{R}}\\
         \mc{O}
     \end{bmatrix},
 \end{equation}
 where $\tilde{\mc{Q}}\in\mathbb{C}^{q\times s\times n}$, $\mc{R}_{11}\in\mathbb{C}^{s\times s\times n}$ is nonsingular and $\tilde{\mc{R}}\in\mathbb{C}^{s\times p\times n}$. If $\mc{S}^{(2)}_{\rg(\mc{T}),\nl(\mc{T})}$ exists then the following are holds:
 \begin{enumerate}
\item[\rm (i)] $\tilde{\mc{R}}*\mc{P}^**\mc{S}*\tilde{\mc{Q}}$ is invertible.
\item[\rm (ii)] $\mc{S}^{(2)}_{\rg(\mc{T}),\nl(\mc{T})}=\tilde{\mc{Q}}*(\tilde{\mc{R}}*\mc{P}^**\mc{S}*\tilde{\mc{Q}})^{-1}*\tilde{\mc{R}}*\mc{P}^*$.
\item[\rm (iii)] $\mc{S}^{(2)}_{\rg(\mc{T}),\nl(\mc{T})}=\mc{S}^{(2)}_{\rg(\tilde{\mc{Q}}),\nl(\tilde{\mc{R}}*\mc{P}^*)}$.
\item[\rm (iv)] $\mc{S}^{(2)}_{\rg(\mc{T}),\nl(\mc{T})}=\tilde{\mc{Q}}*(\mc{Q}_{1}^**\mc{T}*\mc{S}*\mc{Q}_{1})^{-1}*\tilde{\mc{Q}}^**\mc{T}$.
 \end{enumerate}
\end{theorem}
 \begin{proof}
 Let $ \mc{T}*\mc{P}=\mc{Q}*\mc{R}$. Then using the partition of $\mc{Q}$ and $\mc{R}$ we obtain $\mc{T}=\mc{Q}*\mc{R}*\mc{P}^*=\tilde{\mc{Q}}*\tilde{\mc{R}}*\mc{P}^*$. Thus  $\bcr(\mc{T})=\bcr(\tilde{\mc{Q}})\bcr(\tilde{\mc{R}}*\mc{P}^{*})$. It can be easily proved that $\text{rank}(\bcr(\tilde{\mc{Q}}))=sn=\text{rank}(\bcr(\tilde{\mc{R}}*\mc{P}^*))$ and hence $\tilde{\mc{Q}}*\tilde{\mc{R}}*\mc{P}^*$ is full-rank decomposition of $\mc{T}$. Applying Lemma \ref{fl-rank}, we conclude the proof of parts (i), (ii), and (iii).\\
    (iv) It follows by substituting $\tilde{\mc{R}}*\mc{P}=\mc{Q}^**\mc{T}$ in part (ii).
 \end{proof}
\begin{algorithm}[hbt!]
  \caption{Outer inverse using $t$-QR decomposition $\mc{X} = \mc{S}^{(2)}_{\rg({T}), \nl({T})}$} \label{AlgonullQR}
  \begin{algorithmic}[1]
    \Procedure{oinverse}{$\mc{S}$}
    \State \textbf{Input} $ \mc{S}\in\cpqn$, $\mc{T}\in\mathbb{C}^{q\times p\times n}$,  $\rank(\mc{T})=sn \leq rn = \rank(\mc{S})$.
\For{$i \gets 1$ to $n$} 
      \State $\mc{T}=\textup{fft}(\mc{T}, [~ ], i);$
      \EndFor
    \For{$i \gets 1$ to $n$} 
   \State $[\mc{Q}(:,:,i), \mc{R}(:,:,i), \mc{P}(:,:,i)] = \texttt{qr}(\mc{T}(:,:,i))$ 
   \EndFor
\For{$i \gets n$ to $1$} 
      \State $\mc{Q} \leftarrow \textup{ifft}(\mc{Q}, [~ ], i);$,~$\mc{R} \leftarrow \textup{ifft}(\mc{R}, [~ ], i);$, $\mc{P} \leftarrow \textup{ifft}(\mc{P}, [~ ], i);$
      \EndFor
 \State Here $\mc{Q}$ and $\mc{R}$ will be in the following form
 \begin{equation*}
     \mc{Q}=\begin{bmatrix}
      \tilde{\mc{Q}} & \bar{\mc{Q}}   
     \end{bmatrix},~\mc{R}=\begin{bmatrix}
         \bar{\mc{R}} & \bar{\bar{\mc{R}}}\\
         \mc{O} &\mc{O}
     \end{bmatrix}=\begin{bmatrix}
         \tilde{\mc{R}}\\
         \mc{O}
     \end{bmatrix},
 \end{equation*}
 \hspace{.5cm} where $\tilde{\mc{Q}}\in\mathbb{C}^{p\times s\times n}, \tilde{\mc{R}}\in\mathbb{C}^{s\times p\times n}$ are block tensors and $\bar{\mc{R}}\in\mathbb{C}^{s\times s\times n}$ is nonsingular block tensors. 
\State \textbf{compute} $\mc{Y} = \tilde{\mc{R}}*\mc{P}^**\mc{S}*\tilde{\mc{Q}}$ and $\mc{Z} = \tilde{\mc{Q}}^**\mc{T}*\mc{S}*\tilde{\mc{Q}}$

\For{$i \gets 1$ to $n$} 
      \State $\mc{Y}=\textup{fft}(\mc{Y}, [~ ], i)$,~ $\mc{Z}=\textup{fft}(\mc{Z}, [~ ], i);$
      \EndFor
    \For{$i \gets 1$ to $n$} 
   \State $\mc{Y}(:,:,i) = \texttt{inv}(\mc{Y}(:,:,i))$,~~$\mc{Z}(:,:,i) = \texttt{inv}(\mc{Z}(:,:,i))$  
   \EndFor
\For{$i \gets n$ to $1$} 
      \State $\mc{Y} \leftarrow \textup{ifft}(\mc{Y}, [~ ], i)$,~~$\mc{Z} \leftarrow \textup{ifft}(\mc{Z}, [~ ], i);$, 
      \EndFor
\State \textbf{compute}
\begin{eqnarray*}
\mc{S}^{(2)}_{\rg(\mc{T}),\nl(\mc{T})} &=&\tilde{\mc{Q}}*\mc{Y}*\tilde{\mc{R}}*\mc{P}^*.\\
&=&\tilde{\mc{Q}}*\mc{Z}*\tilde{\mc{Q}}^**\mc{T}.
\end{eqnarray*}
   \State \textbf{return } $\mc{S}^{(2)}_{\rg({T}), \nl({T})}$ 
    \EndProcedure
  \end{algorithmic}
\end{algorithm}

\begin{corollary}\label{MCor}
Let $\mc{S}\in\cpqn$, $\mc{T}\in\mathbb{C}^{q\times p\times n}$,  $\rank(\mc{T})=sn \leq rn=\rank(\mc{S})$. Let $\tilde{\mc{Q}}*(\tilde{\mc{R}}*\mc{P}^*)$ be the full-rank decomposition of $\mc{T}$, where $\mc{P}$, $\mc{Q}$, $\mc{R}$, $\tilde{\mc{Q}}$ and $\tilde{\mc{R}}$ as defined in Theorem \ref{qrthm}. Then 
$$\mc{S}^{(2)}_{\rg(\tilde{\mc{Q}}),\nl(\tilde{\mc{R}}*\mc{P}^*)}=\left\{
\begin{array}{ll}
\mc{A}^\dagger, & \mbox{ if }\mc{T}=\mc{S}^*.\\
\mc{A}^D, & \mbox{ if }\mc{T}=\mc{S}^k,\ k\geq ind(\mc{S}),\mbox{ and }\mc{S}\in\mathbb{C}^{p\times p\times n}.
\end{array}
\right.$$
 \end{corollary}

\section{Numerical Examples}
In this section, we have computed the examples by MATLAB, R2022b, which runs on a personal computer with a central processing unit [12-Core Intel Xeon W   3.3 GHz], 96GB of memory, and the macOS Ventura operating system. 
\begin{example}
Let $ \mathcal{T} = \mc{S}^T $ and $\mc{S}\in \mathbb{R}^{3 \times 4 \times 2}$  with frontal slices 
\begin{eqnarray*}
\mc{S}_{1} =
    \begin{pmatrix}
    0   &          -1       &      -1     &        -1\\       
       0       &       1     &        -1     &         1  \\     
       0       &       0       &       0       &       0  
    \end{pmatrix},~
\mc{S}_{2} =
    \begin{pmatrix}
   1        &        1        &        1       &         0     \\  
      -1       &        -1      &          1       &         1\\      
       0       &         0       &         0       &         0 
    \end{pmatrix}.
    \end{eqnarray*}
 With the help of the above Algorithm \ref{AlgonullQR}, we compute the block tensor $\tilde{\mc{Q}} \in \mathbb{R}^{4\times 2 \times 2}$, where the slices are given by

\begin{eqnarray*}
\tilde{\mc{Q}}_1 =
    \begin{pmatrix}
-1/4     &    -379/1676  \\
     195/5042   &   -401/932  \\ 
    1358/2521   &    485/15466 \\
    1358/2521   &  -1549/5542 \\
\end{pmatrix},~
\tilde{\mc{Q}}_2 =
    \begin{pmatrix}
      1/4       &   379/1676 \\ 
    1358/2521    &   108/4907  \\
     195/5042     &  486/619 \\  
     195/5042     & -305/2369 \\ 
 \end{pmatrix}.
    \end{eqnarray*}
Further, using  Corollary \ref{MCor} we compute $\mc{X}=\mc{S}^\dagger = \mc{S}^{(2)}_{\rg(\tilde{\mc{Q}}),\nl(\tilde{\mc{R}}*\mc{P}^*)} = \tilde{\mc{Q}}*(\tilde{\mc{Q}}^**\mc{T}*\mc{S}*\tilde{\mc{Q}})^{-1}*\tilde{\mc{Q}}^**\mc{T}$, where the frontal slices are 
\begin{eqnarray*}
\mc{X}_{1} =
    \begin{pmatrix}
   -2/25     &      1/10    &       0  \\     
      -9/50    &       3/10      &     0    \\   
      -3/25   &       -1/10    &       0  \\     
      -3/25    &       1/10      &     0  
    \end{pmatrix},~
\mc{X}_{2} =
    \begin{pmatrix}
      2/25     &     -1/10   &        0 \\      
      -1/50     &      1/10   &        0   \\    
       3/25     &      1/10   &        0  \\     
      -2/25    &       3/10    &       0  \\ 
    \end{pmatrix}. 
    \end{eqnarray*}
\end{example}

\begin{example}
Let $ \mathcal{T} = \mc{S} $ and $\mc{S}\in \mathbb{R}^{4 \times 4 \times 2}$  with frontal slices 
\begin{eqnarray*}
\mc{S}_{1} =
    \begin{pmatrix}
     2     &         2        &      0       &      -1 \\      
       2      &        4      &        0     &         1 \\      
       0     &         0      &        4    &          1 \\      
      -1     &         1     &         1   &           3   
    \end{pmatrix},~
\mc{S}_{2} =
    \begin{pmatrix}
  0       &      -2         &     0         &    -2  \\     
      -2    &         -4     &         0      &       -1   \\    
       0     &         0     &        -4      &       -1 \\      
      -2     &        -1   &          -1     &         2      
    \end{pmatrix}.
    \end{eqnarray*}
 With the help of the above Algorithm \ref{AlgonullQR}, we compute the block tensor $\tilde{\mc{Q}} \in \mathbb{R}^{4\times 2 \times 2}$, where the slices are given by

\begin{eqnarray*}
\tilde{\mc{Q}}_1 =
    \begin{pmatrix}
-769/1762  &    -332/14353 \\
    -769/881 &      -263/5685  \\
       0      &     1841/1895\\  
    -769/3524  &     263/1137 
\end{pmatrix},~
\tilde{\mc{Q}}_2 =
    \begin{pmatrix}
      0      &  0 \\ 
    0    &   0  \\
     0     & 0 \\  
     0    & 0 
 \end{pmatrix}
    \end{eqnarray*}
Further, using Corollary \ref{MCor} we compute $\mc{X}=\mc{S}^{\#} = \mc{S}^{(2)}_{\rg(\tilde{\mc{Q}}),\nl(\tilde{\mc{R}}*\mc{P}^*)} = \tilde{\mc{Q}}*(\tilde{\mc{Q}}^**\mc{T}*\mc{S}*\tilde{\mc{Q}})^{-1}*\tilde{\mc{Q}}^**\mc{T}$, where the frontal slices are 
\begin{eqnarray*}
\mc{X}_{1} =
    \begin{pmatrix}
   2193/874     &   145/7921   &    -19/7921   &   2080/1383 \\ 
    2193/437     &   290/7921   &    -38/7921  &    1513/503  \\ 
     416/417     &   -38/7921   &    442/7921  &     556/549 \\  
    2080/1383   &     63/7921   &    101/7921  &     971/966  
    \end{pmatrix},
     \end{eqnarray*}
     \begin{eqnarray*}
\mc{X}_{2} =
    \begin{pmatrix}
      2177/874   &    -145/7921    &    19/7921   &   2069/1383 \\ 
    2177/437    &   -290/7921    &    38/7921    &  1505/503   \\
     418/417    &     38/7921   &   -442/7921   &    542/549  \\ 
    2069/1383   &    -63/7921  &    -101/7921  &     961/966  
    \end{pmatrix}. 
    \end{eqnarray*}
\end{example}
The errors associated to different matrix and tensor equations are represented in the below Table \ref{tab:error}.
\begin{table}[htb!]
    \centering
    \caption{Errors associated to matrix and tensor equations}
    \vspace*{0.2cm}
    \renewcommand{\arraystretch}{1.2}
    \begin{tabular}{ccc}
    \hline
     $\mc{E}^t_1 = \|\mc{S}-\mc{S}*\mc{X}*\mc{S}\|_F$  & $\mc{E}^t_{2} = \|\mc{X}-\mc{X}*\mc{S}*\mc{X}\|_F$ & ${E}^{t}_{3} =\|\mc{S}*\mc{X}-(\mc{S}*\mc{X})^T\|_F$ \\     
    $\mc{E}^{t}_{4} =\|\mc{X}*\mc{S}-(\mc{X}*\mc{S})^T\|_F$&$\mc{E}^{t}_{5} = \|\mc{S}*\mc{X}-\mc{X}*\mc{S}\|_F$
       &$\mc{E}^{t}_{1^k} = \|\mc{X}*\mc{S}^{k} -\mc{S}^k\|_F$ \\      
       ${E}^m_1 = \|{S}-{S}{X}{S}\|_F$  &${E}^m_{2} = \|{X}-{X}{S}{X}\|_F$&${E}^{m}_{3} =\|{S}{X}-({S}{X})^T\|_F$ \\
       ${E}^{m}_{4} =\|{X}{S}-({X}{S})^T\|_F$&${E}^{m}_{5} = \|{S}{X}-XS\|_F$
       &${E}^{m}_{1^k} = \|XS^{k+1} -{S}^k\|_F$ \\
         \hline
             \end{tabular}
       \label{tab:error}
\end{table}

Next, we have compared the errors and mean CPU time for computing the the Moore-Penrose inverse and Drazin inverse by using $t$-QR algorithm \ref{AlgonullQR}. It is observed that for large matrices, the tensor structured computation are more efficient than matrix based computations as shown in Tables \ref{tab:error-chow}-\ref{tab:error-gearmat}. In Tables \ref{tab:error-chow}-\ref{tab:error-gearmat}, MT$^t$ and MT$^m$ respectively represents the mean CPU time for tensor and matrix  structured computations. 
    \begin{table}[H] 
        \begin{center}
       \caption{Error analysis and computation mean time for $chow$, when $\mc{X}=\mc{S}^{\dagger}$ and $X=S^{\dagger}$}
          \renewcommand{\arraystretch}{1.2}
    \begin{tabular}{|c|c|c|c|c|c|}
    \hline
        Size of $\mc{S}$ & Size of $S$ & MT$^t$ & MT$^m$ &Error$^t$ &Error$^m$\\ 
        \hline
    \multirow{4}{*}{$400\times 400\times 400$} & \multirow{4}{*}{$8000\times 8000$}&  \multirow{4}{*}{25.50} &\multirow{4}{*}{35.00}&
    $\mc{E}^{t}_{1}= 1.157e^{-10}$  & 
    ${E}^{m}_{1}=1.591e^{-09}$ \\     
& &  & &$\mc{E}^{t}_{2}= 1.046e^{-16}$  
          &${E}^{m}_{2}=9.898e^{-11}$ \\
& &  & &$\mc{E}^{t}_{3}=4.440e^{-14}$  
       &${E}^{m}_{3}=1.029e^{-10}$ \\
& &  & &$\mc{E}^{t}_{4}=5.645e^{-13}$  
       &${E}^{m}_{4}=3.126e^{-09}$ \\
    \hline
    \multirow{4}{*}{$500\times 500\times 400$} & \multirow{4}{*}{$10000\times 10000$}&  \multirow{4}{*}{41.26} &\multirow{4}{*}{64.48}&
    $\mc{E}^{t}_{1}= 3.121e^{-10}$  & 
    ${E}^{m}_{1}=1.577e^{-09}$ \\     
& &  & &$\mc{E}^{t}_{2}= 2.146e^{-16}$  
&${E}^{m}_{2}=1.561e^{-10}$ \\& &  & &
$\mc{E}^{t}_{3}=9.100e^{-14}$  
&${E}^{m}_{3}=1.642e^{-10}$ \\
& &  & &$\mc{E}^{t}_{4}=1.552e^{-12}$  
&${E}^{m}_{4}=3.115e^{-09}$ \\
    \hline
    \multirow{4}{*}{$800\times 800\times 400$} & \multirow{4}{*}{$16000\times 16000$}&  \multirow{4}{*}{115.44} &\multirow{4}{*}{244.64}&
    $\mc{E}^{t}_{1}=4.654e^{-10}$  & ${E}^{m}_{1}=4.569e^{-09}$ \\     
& &  & &$\mc{E}^{t}_{2}= 3.306e^{-16}$  &${E}^{m}_{2}=2.994e^{-10}$ \\& &  & &$\mc{E}^{t}_{3}=1.333e^{-13}$  &${E}^{m}_{3}=3.176e^{-10}$ \\
& &  & &$\mc{E}^{t}_{4}=2.252e^{-12}$  &${E}^{m}_{4}= 9.068e^{-09}$ \\
       \hline
    \end{tabular}
    \label{tab:error-chow}  
          \end{center}
\end{table}

    \begin{table}[H]
    \begin{center}
       \caption{Error analysis and computation mean time for $kahan$, when $\mc{X}=\mc{S}^{\dagger}$ and $X=S^{\dagger}$}
          \renewcommand{\arraystretch}{1.2}
    \begin{tabular}{|c|c|c|c|c|c|}
    \hline
        Size of $\mc{S}$ & Size of $S$ & MT$^t$ & MT$^m$ &Error$^t$ &Error$^m$\\ 
        \hline
    \multirow{4}{*}{$400\times400\times 400$} & \multirow{4}{*}{$8000\times 8000$}&  \multirow{4}{*}{26.43} &\multirow{4}{*}{33.07}&
    $\mc{E}^{t}_{1}= 6.195e^{-12}$  & 
    ${E}^{m}_{1}=1.485e^{-09}$ \\     
& &  & &$\mc{E}^{t}_{2}=3.638e^{-07}$  
          &${E}^{m}_{2}=1.885e^{-02}$ \\
& &  & &$\mc{E}^{t}_{3}=6.209e^{-05}$  
       &${E}^{m}_{3}=1.070e^{-02}$ \\
& &  & &$\mc{E}^{t}_{4}=3.207e^{-14}$  
       &${E}^{m}_{4}=3.529e^{-01}$ \\
    \hline
    \multirow{4}{*}{$500\times 500\times 400$} & \multirow{4}{*}{$10000\times 10000$}&  \multirow{4}{*}{43.20} &\multirow{4}{*}{64.67}&
    $\mc{E}^{t}_{1}= 8.208e^{-11}$  & ${E}^{m}_{1}=2.076e^{-09}$ \\     
& &  & &$\mc{E}^{t}_{2}= 9.919e^{-07}$  &${E}^{m}_{2}=1.796e^{-02}$ \\& &  & &$\mc{E}^{t}_{3}=2.182e^{-04}$  &${E}^{m}_{3}=1.011e^{-02}$ \\
& &  & &$\mc{E}^{t}_{4}=1.732e^{-02}$  &${E}^{m}_{4}=3.530e^{-01}$ \\
       \hline
    \end{tabular}
       \label{tab:error-kahan}  
    \end{center}
\end{table}

\begin{table}[H]
    \begin{center}
       \caption{Error analysis and computation mean time for $cycol$, when $\mc{X}=\mc{S}^{\#}$ and $X=S^{\#}$}
          \renewcommand{\arraystretch}{1.2}
    \begin{tabular}{|c|c|c|c|c|c|}
    \hline
        Size of $\mc{S}$ & Size of $S$ & MT$^t$ & MT$^m$ &Error$^t$ &Error$^m$\\ 
        \hline
    \multirow{3}{*}{$400\times 400\times 400$} & \multirow{3}{*}{$8000\times 8000$}&  \multirow{3}{*}{32.61} &\multirow{3}{*}{46.11}&
    $\mc{E}^{t}_{1}= 1.449e^{-08}$  & 
    ${E}^{m}_{1}=3.248e^{-07}$ \\     
& &  & &$\mc{E}^{t}_{2}= 1.536e^{-10}$  
          &${E}^{m}_{2}=9.764e^{-09}$ \\
& &  & &$\mc{E}^{t}_{5}=2.816e^{-10}$  
       &${E}^{m}_{5}=4.263e^{-08}$ \\
    \hline
    \multirow{3}{*}{$500\times 500\times 400$} & \multirow{3}{*}{$10000\times 10000$}&  \multirow{3}{*}{47.02} &\multirow{3}{*}{91.00}&
    $\mc{E}^{t}_{1}= 3.086e^{-08}$  & 
    ${E}^{m}_{1}=1.0552e^{-05}$ \\     
& &  & &$\mc{E}^{t}_{2}= 6.351e^{-10}$  
          &${E}^{m}_{2}=4.645e^{-07}$ \\
& &  & &$\mc{E}^{t}_{5}=4.614e^{-10}$  
       &${E}^{m}_{5}=1.527e^{-06}$ \\
       \hline
    \end{tabular}
       \label{tab:error-magic}  
    \end{center}
\end{table}
\begin{table}[H]
    \begin{center}
       \caption{Error analysis and computation mean time for $gearmat$, when $\mc{X}=\mc{S}^{D}$ and $X=S^{D}$}
          \renewcommand{\arraystretch}{1.2}
     \begin{tabular}{|c|c|c|c|c|c|}
    \hline
        Size of $\mc{S}$ & Size of $S$ & MT$^t$ & MT$^m$ &Error$^t$ &Error$^m$\\ 
        \hline
    \multirow{3}{*}{$400\times 400\times 400$} & \multirow{3}{*}{$8000\times 8000$}&  \multirow{3}{*}{22.35} &\multirow{3}{*}{36.42}&
    $\mc{E}^{t}_{1^2}= 1.120e^{-06}$  & 
    ${E}^{m}_{1^2}=3.301e^{-07}$ \\     
& &  & &$\mc{E}^{t}_{2}= 1.389e^{-13}$  
          &${E}^{m}_{2}=5.852e^{-04}$ \\
& &  & &$\mc{E}^{t}_{5}=2.070e^{-10}$  
       &${E}^{m}_{5}=7.704e^{-04}$ \\
    \hline
    \multirow{3}{*}{$500\times 500\times 400$} & \multirow{3}{*}{$10000\times 10000$}&  \multirow{3}{*}{38.15} &\multirow{3}{*}{75.94}&
    $\mc{E}^{t}_{1^2}= 2.577e^{-06}$  & 
    ${E}^{m}_{1^2}=6.151e^{-07}$ \\     
& &  & &$\mc{E}^{t}_{2}= 9.467e^{-13}$  
          &${E}^{m}_{2}=3.001e^{-03}$ \\
& &  & &$\mc{E}^{t}_{5}=4.332e^{-10}$  
       &${E}^{m}_{5}=1.734e^{-03}$ \\
       \hline
    \end{tabular}
       \label{tab:error-gearmat}  
    \end{center}
\end{table}

Tensor QR factorization is created by applying the appropriate matrix QR factorization, which is discussed the above section, i.e., 
\begin{eqnarray*}
    \mc{A}=\mc{Q}*\mc{R} \Longleftrightarrow \bcr(\mc{A})=\bcr(\mc{Q})~\bcr(\mc{R})
\end{eqnarray*}
It is worth mentioning that the tensor QR factorization algorithm \cite{kilmer11} and the recent randomized sampling technique \cite{Duersch17} motivated us to discuss randomized tensor-based QR factorization algorithms. (See Algorithm~\ref{AlgRQR}). The main features of the algorithm are its high efficiency and easy parallelization. 
Recall the randomized QR decomposition (with column pivoting) of the matrices. Let $\mc{A}$ be $m \times n$ and $k$ be the desired factorization rank, that is, $k \leq \min (m, n)$. Here we use Algorithm 4 of \cite{Duersch17}, i.e.,``Randomized QR with Column Pivoting.''  The matrix computation Matlab function (\texttt{rqrcp}) is utilized, i.e., $[Q, R, P]=\texttt{rqrcp}(\mc{A},k).$ Following this we consider $\mc{A} \in \mathbb{R}^{p\times q \times n}$ and $k=\rank(\mc{A})=\text{rank}(\bcr(\mc{A}))$. 
\begin{eqnarray*}
 {\bcr}
     (\mc{A})  &=&
 (F_{n}^* \otimes I_{q})\cdot 
    \begin{pmatrix}
     {A}_1 &  &  & \\
     & {A}_2 &  & \\
     & &  \ddots &\\ 
    & &  & {A}_{n}
    \end{pmatrix}  \cdot (F_{n} \otimes I_{p})\\
&=& (F_{n}^* \otimes I_{q})\cdot 
    \begin{pmatrix}
     {(QRP)}_1 &  &  & \\
     & {(QRP)}_2 &  & \\
     & &  \ddots &\\ 
    & &  & {(QRP)}_{n}
    \end{pmatrix}  \cdot (F_{n} \otimes I_{p})\\  
  &=& {\bcr}(\mc{Q}) {\bcr}(\mc{R}) {\bcr}(\mc{P}).
\end{eqnarray*}
\begin{algorithm}[hbt!]
  \caption{Randomized $t$-QR Algorithm (R-$t$-QR)} \label{AlgRQR}
  \begin{algorithmic}[1]
    \Procedure{RQR}{$\mc{S}$}
    \State \textbf{Input} $ \mc{S}\in\cpqn$.
\For{$i \gets 1$ to $n$} 
      \State $\mc{S}=\textup{fft}(\mc{S}, [~ ], i);$
      \EndFor
    \For{$i \gets 1$ to $n$} 
   \State $[\mc{Q}(:,:,i), \mc{R}(:,:,i), \mc{P}(:,:,i)] = \texttt{rqrcp}(\mc{S}(:,:,i))$ 
   \EndFor
   \EndProcedure
  \end{algorithmic}
\end{algorithm}

  \begin{table}[H] 
        \begin{center}
       \caption{Time comparison with Randomized $t$-QR for $cycol$, when $\mc{X}=\mc{S}^{\dagger}$}
          \renewcommand{\arraystretch}{1.2}
    \begin{tabular}{|c|c|c|c|c|}
    \hline
        Size of $\mc{S}$ &  MT$^{t-QR}$ & MT$^{R-t-QR}$ &Error$^{t-QR}$ &Error$^{R-t-QR}$\\ 
        \hline
    \multirow{4}{*}{$400\times 400\times 400$} &  \multirow{4}{*}{43.01} &\multirow{4}{*}{33.19}&
    $\mc{E}^{t}_{1}= 6.630e^{-12}$  & $\mc{E}^{t}_{1}=6.634e^{-12}$ \\     
&   & &$\mc{E}^{t}_{2}= 1.046e^{-17}$  &$\mc{E}^{t}_{2}=1.4674e^{-17}$ \\
&   & &$\mc{E}^{t}_{3}=1.476e^{-14}$   &$\mc{E}^{t}_{3}=1.607e^{-14}$ \\
&   & &$\mc{E}^{t}_{4}=2.205e^{-14}$  &$\mc{E}^{t}_{4}=1.339e^{-13}$ \\
    \hline
    \multirow{4}{*}{$500\times 500\times 500$} &   \multirow{4}{*}{83.34} &\multirow{4}{*}{65.80}&
    $\mc{E}^{t}_{1}= 1.036e^{-12}$  &     $\mc{E}^{t}_{1}=1.038e^{-11}$ \\     
&   & &$\mc{E}^{t}_{2}= 1.387e^{-17}$ &$\mc{E}^{t}_{2}=1.456e^{-17}$ \\&  & &
$\mc{E}^{t}_{3}=1.736e^{-14}$  &$\mc{E}^{t}_{3}=1.965e^{-14}$ \\
&   & &$\mc{E}^{t}_{4}=2.309e^{-14}$ &$\mc{E}^{t}_{4}=1.830e^{-13}$ \\
         \hline
    \end{tabular}
    \label{tab:rqr-cycol}  
          \end{center}
\end{table}

Table~\ref{tab:error-rqq-cycol-group} represents both the accuracy and low running time for randomized algorithm \ref{AlgRQR}. We have seen that the randomized tensor QR decomposition for computing the outer inverse, specifically the Moore–Penrose and Group inverses, performs better in terms of running time and residual errors. Further, if we choose the appropriate $k$, it may perform better. The errors and computational running time for the matlab function cycol has been worked out numerically. However, one can use any other matrices with appropriate $k \leq \rank(\mc{A})$. For choosing suitable value $k$ one can see \cite{zhu2022tensor}. 
  \begin{table}[H] 
        \begin{center}
       \caption{Time comparison with Randomized $t$-QR for $cycol$, when $\mc{X}=\mc{S}^{\#}$}
          \renewcommand{\arraystretch}{1.2}
    \begin{tabular}{|c|c|c|c|c|}
    \hline
        Size of $\mc{S}$ &  MT$^{t-QR}$ & MT$^{R-t-QR}$ &Error$^{t-QR}$ &Error$^{R-t-QR}$\\ 
        \hline
    \multirow{4}{*}{$400\times 400\times 400$} &  \multirow{4}{*}{32.54} &\multirow{4}{*}{27.19}&
    $\mc{E}^{t}_{1}= 2.631e^{-08}$  &     $\mc{E}^{t}_{1}=3.212e^{-08}$ \\     
&   & &$\mc{E}^{t}_{2}= 6.867e^{-10}$    &$\mc{E}^{t}_{2}=1.304e^{-09}$ \\
&   & &$\mc{E}^{t}_{5}=8.519e^{-10}$    &$\mc{E}^{t}_{5}=3.104e^{-09}$ \\
    \hline
    \multirow{4}{*}{$500\times 500\times 500$} &   \multirow{4}{*}{64.03} &\multirow{4}{*}{56.11}&
    $\mc{E}^{t}_{1}=6.128 e^{-08}$  & $\mc{E}^{t}_{1}=6.344e^{-08}$ \\     
&   & &$\mc{E}^{t}_{2}= 1.622e^{-09}$  &$\mc{E}^{t}_{2}=2.364e^{-09}$ \\&  & &
$\mc{E}^{t}_{5}=8.886^{-10}$  &$\mc{E}^{t}_{5}=1.783e^{-09}$ \\
         \hline
    \end{tabular}
    \label{tab:error-rqq-cycol-group}  
          \end{center}
\end{table}

\newpage


\section{Conclusion}
The outer inverse of a tensor with specified range and nullity conditions are developed and analysed. Appropriate algorithms are designed and the computation of the outer inverse is computed with the help of the tensor QR decomposition. Subsequently, we discuss the relation of this outer inverse with other generalized inverses such as Moore–Penrose inverse, group inverse, and Drazin inverse. The generality of the developed algorithm is reflected in the capability that many kinds of other generalized inverses can be obtained as appropriate special cases. This gives us the flexibility to choose generalized inverses depending on applications. We have worked out a few numerical examples to illustrate the QR algorithm in both tensor and matrix structured domain. The computations \cite{Duersch2020randomized}  in tensor framework under $t$-product is more efficient as observed in the examples.
\section*{ Author contribution}
The three authors contributed equally to this work.

\section*{Data Availability Statement}
Data sharing is not applicable to this article as no datasets were generated or analyzed during the current study.


\section*{Funding}
\begin{itemize}
    \item Ratikanta Behera is grateful for the support of the Science and Engineering Research Board (SERB), Department of Science and Technology, India, under Grant No. EEQ/2022/001065. 
\item Jajati Keshari Sahoo is grateful for the support of the Science and Engineering Research Board (SERB), Department of Science and Technology, India, under Grant No. SUR/2022/004357. 
\item Yimin Wei is grateful for the support of  the Shanghai Municipal Science and Technology Commission under grant 23WZ2501400 and the Ministry
of Science and Technology of China under grant G2023132005L.

\end{itemize}

\section*{Declarations}

Ethical approval.  Not applicable.\\
Conflicts of Interest.  The authors declare no competing interests.

\newpage

\bibliographystyle{abbrv}
\bibliography{References}

\end{document}